\documentclass[letterpaper, 10 pt, conference]{ieeeconf}  
\usepackage{cite}
\usepackage{amsmath,amssymb,amsfonts}
\usepackage{graphicx}
\usepackage{algorithm}
\usepackage{algpseudocode}
\usepackage{hyperref}
\hypersetup{hidelinks}
\usepackage{textcomp}

\IEEEoverridecommandlockouts                              
\allowdisplaybreaks

\usepackage{nomencl}
\makenomenclature

\newtheorem{theorem}{Theorem}
\newtheorem{lemma}{Lemma}

\newtheorem{corollary}{Corollary}
\newtheorem{assumption}{Assumption}

\newcommand{\ml}{}

\newcommand{\mll}{}

\usepackage{subcaption}

\author{Mengmou~Li,
Yu~Zhou, Xun~Shen, and Masaaki~Nagahara 
\thanks{M.~Li, Y.~Zhou and M.~Nagahara are with the Graduate School of Advanced Science and Engineering, Hiroshima University, Higashi-Hiroshima City, Japan (e-mail: mmli.research@gmail.com; yuzhou@hiroshima-u.ac.jp; nagahara@ieee.org). M.~Li is supported by JSPS KAKENHI under Grant Numbers 24K23864 and 26K17401. This work is also supported by Japan Science and Technology Agency (JST) as part of Adopting Sustainable Partnerships for Innovative Research Ecosystem (ASPIRE), Grant Number JPMJAP2402.}
\thanks{X.~Shen is with the Graduate School of Engineering, Tokyo University of Agriculture and Technology, Tokyo, Japan.}
}

\title{\LARGE \bf A Canonical Structure for Constructing Projected First-Order Algorithms With Delayed Feedback}

\begin{document}

\maketitle
\begin{abstract}
This work introduces a canonical structure for a broad class of unconstrained first-order algorithms that admit a Lur’e representation, including systems with relative degree greater than one, e.g., systems with delayed gradient feedback. The proposed canonical structure is obtained through a simple linear transformation. It enables a direct extension from unconstrained optimization algorithms to set-constrained ones through projection in a Lyapunov-induced norm. The resulting projected algorithms attain the optimal solution while preserving the convergence rates of their unconstrained counterparts.
\end{abstract}


%
\IEEEpeerreviewmaketitle


\section{Introduction}
First-order optimization algorithms play an important role in many fields such as machine learning \cite{gower2020variance,li2020accelerated}, and power systems \cite{molzahn2017survey,li2023distributed}, owing to their simple implementation and scalability to large-scale problems \cite{nesterov2003introductory,boyd2004convex}.
For strongly convex optimization problems, a major research topic is to construct and characterize algorithms with fast linear convergence rates \cite{d2021acceleration}.
In recent years, control-theoretic methods have provided a powerful viewpoint for this purpose by representing optimization algorithms as feedback interconnections. 
In particular, integral quadratic constraints (IQCs) \cite{megretski1997system,seiler2014stability} offer a unified framework for convergence analysis in both frequency and time domains \cite{lessard2016analysis,zhang2022zames}. 
These ideas have also led to convex synthesis methods for accelerated algorithms \cite{scherer2021convex,scherer2023optimization,scherer2025tutorial}, as well as tight analysis of specific methods such as the triple momentum algorithm \cite{van2017fastest}.
Such an application is successful thanks to recent advances of the IQC analysis, including the $\rho$-IQC \cite{lessard2016analysis,boczar2017exponential,hu2016exponential}, the O'Shea--Zames--Falb and Popov multipliers \cite{kulkarni2002all,carrasco2013equivalence,carrasco2016zames,fetzer2017absolute,carrasco2019convex,van2022absolute,su2023necessity,gyotoku2024dual,li2024generalization}.
While IQCs are employed in the above works to characterize nonlinearities, a parallel line of research has analyzed optimization algorithms through interpolation conditions for various classes of nonlinear mappings \cite{taylor2017exact,taylor2017smooth,taylor2017performance}. 

However, although projected and proximal algorithms have also been analyzed or synthesized in the literature \cite{taylor2017exact,lessard2022analysis,li2025convergence,miller2026structure}, there has been little work on deriving constrained algorithms directly from unconstrained ones in a way that preserves their convergence guarantees. A recent exception is \cite{li2025first}, which proposed a framework for constructing projected first-order algorithms from unconstrained ones in Lur'e form while preserving the same linear convergence rate bounds. The key idea is to use projection in the norm induced by a quadratic Lyapunov function of the unconstrained dynamics. This result provides a direct bridge between unconstrained algorithm analysis and constrained algorithm design.
Nevertheless, the framework in \cite{li2025first} does not cover algorithmic systems whose relative degree with respect to the gradient input is greater than one, e.g., systems with delayed gradient feedback or channel dynamics. Such structures arise naturally in various scenarios, such as the asynchronous gradient descent and distributed optimization \cite{zheng2017asynchronous,choi2024non,hatanaka2018passivity,li2020smooth}.
Moreover, the class of admissible multipliers for constructing the stepwise quadratic Lyapunov functions is not addressed in \cite{li2025first}.

This paper extends \cite{li2025first} to first-order algorithms with relative degree more than one. Our starting point is a general Lur'e-type representation of the unconstrained algorithm.
We show that, after a suitable state transformation, such algorithms can be rewritten in a standard form we term the \emph{canonical structure}, which isolates the state directly constrained by the optimization variable and makes the delayed gradient feedback explicit.
Based on this structure, we construct projected algorithms by applying a projection in the Lyapunov-induced norm.
We show that if the original unconstrained algorithm admits a quadratic Lyapunov function certifying a linear convergence rate bound, then the associated projected algorithm retains the same rate bound.

The contributions of this paper are summarized as follows.
\begin{enumerate}
    \item We show that a broad class of first-order algorithms in the Lur'e form, including those with relative degree greater than one with respect to the gradient input, can be transformed into a canonical structure that allows a ready constrained extension.
    \item We introduce a lifted-signal representation for the unconstrained algorithmic dynamics and discuss the associated pointwise multipliers within the IQC framework, leading to stepwise Lyapunov functions.
    \item Based on this canonical structure, we propose projected algorithms using projection in the Lyapunov-induced norm, and prove that they achieve the same linear convergence rate bounds as their unconstrained counterparts.
\end{enumerate}
The significance of the proposed framework is that it reduces the constrained design problem to the unconstrained one: once a linear convergence rate bound has been established for the unconstrained algorithm through a quadratic Lyapunov function, the corresponding projected algorithm can be constructed without changing the original algorithmic parameters.
This is particularly appealing for higher-order and synthesized algorithms, for which direct constrained analysis is often difficult.

The rest of the paper is organized as follows.
In Section~\ref{Preliminaries}, we review some preliminaries from convex analysis. 
In Section~\ref{sec: unconstrained first-order algorithms}, we construct the canonical structure of the first-order algorithms and the IQC analysis.
In Section~\ref{sec: Projected Algorithm and Analysis}, Lyapunov-induced projected algorithms are proposed and proved theoretically with the same rate bounds as the unconstrained ones.
A numerical example is presented in Section~\ref{Numerical Example}.
Finally, the paper is concluded in Section~\ref{Conclusion}.

\section{Preliminaries}\label{Preliminaries}
$\mathbb{N}$, $\mathbb{R}$, $\mathbb{C}$, $\mathbb{R}^{n}$, $\mathbb{R}^{n \times m}$ denote the sets of natural numbers, real numbers, complex numbers, $n$-dimensional real vectors, and $n\times m$ real matrices, respectively.
The $n \times n$ identity matrix, $n \times n$, and $ n \times m$ zero matrices are denoted by $I_{n}$, $0_{n}$, and $0_{n \times m}$, respectively. Subscripts are omitted when they can be inferred from the context.
The diagonal and block diagonal operator is represented by $\operatorname{diag} (\cdot)$, $\operatorname{blkdiag}(\cdot)$, respectively. The $n$-dimensional all-ones vector is denoted by $\mathbf{1}_{n}$.
The Kronecker product is denoted by $\otimes$. 
The range of a matrix $A$ is denoted by $\text{im} (A)$.
The Moore–Penrose inverse of matrix $A$ is denoted by $A^{\dag}$. Given a positive definite matrix $P$, the $P$ weighted norm is given by $\| x\|_{P} = \sqrt{x^\top P x}$.
We use $x_{*}$ to denote a reference point of $x$ or a fixed point of an operator acting on $x$. 

\subsection{Convex analysis}
The convex optimization problem considered in this work is given by 
\begin{align}\label{eq: constrained convex optimization problem}
    & \underset{y \in \Omega \subseteq \mathbb{R}^{d}}{\text{minimize}} ~ f(y)
\end{align}
where the objective function $f:\mathbb{R}^{d} \to \mathbb{R}$ is assumed to be differentiable and convex, the constraint set $\Omega \subseteq \mathbb{R}^{d}$ is a closed and convex. When $\Omega = \mathbb{R}^{d}$, \eqref{eq: constrained convex optimization problem} reduces to an unconstrained problem.

We consider objective function $f \in S(m, L)$, which denotes the class of differentiable, \textit{$m$-strongly convex}, and \textit{$L$-Lipschitz smooth} functions with $L \geq m > 0$, that is, for all $x,~ y \in \mathbb{R}^{d}$,
\begin{align*}
    m \|x - y \|^2_2 \leq \left( \nabla f(x) - \nabla f (y)\right)^{\!\top} \!(x - y) \leq L \| x - y\|^2_2.
\end{align*}

A map $\phi : \mathbb{R}^d \to \mathbb{R}^d$ is said to be \textit{slope-restricted} in $[\alpha,  \beta]$ with $\alpha \leq \beta \in \mathbb{R}$, if $\left( \phi(x) - \phi(y) - \alpha (x - y) \right)^{\!\top} \! \left( \phi(x) - \phi(y) - \beta (x - y) \right) \leq 0$, for all $x, y \in \mathbb{R}^{d}$.
Given $f \in S(m, L)$, $\nabla f$ is slope-restricted in $[m, L]$.

Given a closed and convex set $\Omega$, its \textit{normal cone} at a point $x\in \Omega$ is defined by
$N_{\Omega} (x) = \left\{ v: v^{\top} ( y - x ) \leq 0, ~ \forall  y \in \Omega \right\}$.
\ml{The optimal solution $y_*$ to problem \eqref{eq: constrained convex optimization problem} is given by \cite{ruszczynski2011nonlinear},
\begin{align}\label{eq: optimal solution to the constrained problem}
    - \nabla f(y_*) \in N_{\Omega} (y_*).
\end{align}

The projection of a point $x$ onto a closed convex set $\Omega$ with respect to the norm $\| \cdot \|_{P}$ is defined by 
\begin{align}\label{eq: definition of P norm projection}
    \Pi_{\Omega}^{P} (x) = \underset{y \in \Omega}{\operatorname{argmin}} \| x - y \|_{P}^2.
\end{align}
A projection is \textit{non-expansive} in the associated norm \cite{bauschke2017convex}, i.e.,
\begin{align}\label{eq: nonexpansiveness of projection}
    \left\| \Pi_{\Omega}^{P} (x) - \Pi_{\Omega}^{P} (y) \right\|_{P} \leq  \| x - y\|_{P},~\text{for all~} x, y.
\end{align}
When $P = I$, the projection in \eqref{eq: definition of P norm projection} reduces to the standard projection in the Euclidean norm, denoted by $\Pi_{\Omega} (\cdot)$.
The normal cone $N_{\Omega} (x)$ for $x \in \Omega$ can also be written as
$N_{\Omega} (x) = \left\{ v : \Pi_{\Omega} (x + v) = x\right\}$.}

\section{Unconstrained First-order algorithms and Canonical Structure}\label{sec: unconstrained first-order algorithms}
\subsection{Canonical Structure}
Fixed-stepsize First-order discrete-time algorithms for optimization problem \eqref{eq: constrained convex optimization problem}, including gradient descent and Nesterov's methods, can be reformulated into a Lur'e system, where a discrete-time linear time-invariant (LTI) system is interconnected with a nonlinear operator \cite{lessard2016analysis}. The LTI system is given by
\begin{align}\label{eq: LTI system}
    \xi_{k + 1} =& A \xi_{k} + B u_{k}, ~~ \xi_{0} \in \mathbb{R}^{nd}, \nonumber\\
    y_{k} =& C \xi_{k}
\end{align}
where $A \in \mathbb{R}^{n d \times n d}$,  $B \in \mathbb{R}^{n d \times d}$,  $C\in \mathbb{R}^{ d \times n d}$, and the input $u_{k} \in \mathbb{R}^{d}$ is given by the gradient of the output $y_{k} \in \mathbb{R}^{d}$,
\begin{align}\label{eq: gradient input to u_k}
    u_{k} = \nabla f(y_{k}).
\end{align}
As they are first-order algorithms, all the matrices above admit the so-called dimensionality reduction \cite[Sec.~4.2]{lessard2016analysis}, i.e., they can be rewritten as a Kronecker product form $(\cdot) \otimes I_d$.
The \textit{relative degree} of LTI system \eqref{eq: LTI system} is the smallest $r$ such that $C A^{r - 1} B \neq 0$, that is, the output $y_{k + r}$ explicitly depends on the input $u_k$.
Moreover, according to \cite{scherer2021convex}, system \eqref{eq: LTI system} can be written as
\begin{align}\label{eq: state-space system observable form}
\begin{bmatrix}
\begin{array}{c}
\xi_{k+1} \\ \hline y_{k}
\end{array}
\end{bmatrix}
= 
\begin{bmatrix}
\begin{array}{c}
\xi^{(1)}_{k+1} \\ \xi^{(2)}_{k+1} \\ \hline y_{k}
\end{array}
\end{bmatrix}
= 
    \begin{bmatrix}
    \begin{array}{cc|c}
         I_d & 0 & I_d\\
         \mll{A_{21}} & \mll{A_{22}} & 0\\
         \hline
         C_{1} & C_{2} & 0
    \end{array}
    \end{bmatrix}
    \begin{bmatrix}
        \begin{array}{c}
             \xi^{(1)}_{k} \\ \xi^{(2)}_{k} \\ \hline u_{k}
        \end{array}
    \end{bmatrix}
\end{align}
where $\xi^{(1)}_{k} \in \mathbb{R}^{d}$, $\xi^{(2)}_{k} \in \mathbb{R}^{(n-1)d}$, for some $n > 1$.

Next, we present another form that we call the \textbf{canonical structure} of first-order algorithms for unconstrained optimization, as it will become clear later that this form allows for an easy extension of algorithms to their projected versions.

For system \eqref{eq: state-space system observable form} with relative degree $r$, we can formulate their canonical structure, summarized by the following theorem.
\begin{theorem}\label{eq: canonical form of LTI system}
System \eqref{eq: LTI system} with relative degree $r \geq 1$ can be rewritten as
\begin{align}\label{eq: canonical structure}
\begin{bmatrix}
\begin{array}{c}
\tilde{\xi}^{(1)}_{k+1} \\ \tilde{\xi}^{(2)}_{k+1} \\ \hline y_{k}
\end{array}
\end{bmatrix}
= &
\begin{bmatrix}
\begin{array}{c}
y_{k + r} \\ \vdots \\ y_{k + 1} \\ \tilde{\xi}^{(2)}_{k+1} \\ \hline y_{k}
\end{array}
\end{bmatrix}
= 
\begin{bmatrix}
    \begin{array}{c|c}
    \tilde{A} & \tilde{B} \\
         \hline
         \tilde{C} & 0
    \end{array}
\end{bmatrix}
    \begin{bmatrix}
        \begin{array}{c}
        y_{ k+r -1 } \\ \vdots \\ y_{k} \\ \tilde{\xi}^{(2)}_{k} \\ \hline u_{k}
        \end{array}
    \end{bmatrix} \nonumber\\
= &
    \begin{bmatrix}
    \begin{array}{cc|c}
         \tilde{A}_{11} & \tilde{A}_{12} & \tilde{B}_1 \\
         \tilde{A}_{21} & \tilde{A}_{22} & 0\\
         \hline
         \tilde{C}_1 & 0 & 0
    \end{array}
    \end{bmatrix}
    \begin{bmatrix}
        \begin{array}{c}
         y_{k + r - 1} \\ \vdots \\ y_{k} \\ \tilde{\xi}^{(2)}_{k} \\ \hline u_{k}
        \end{array}
    \end{bmatrix}
\end{align}
where $r \geq 1$ denotes the relative degree of system \eqref{eq: LTI system}, 
\begin{align*}
& \tilde A_{11}=
\begin{bmatrix}
\star & \star & \cdots & \star\\
I_d & 0 & \cdots & 0\\
\vdots & \ddots & \ddots & \vdots\\
0 & \cdots & I_d & 0
\end{bmatrix}\in\mathbb{R}^{rd\times rd},\\
& \tilde A_{12} = 
\begin{bmatrix}
    \star \\ 0 \\ \vdots \\ 0 
\end{bmatrix} \in \mathbb{R}^{rd \times (n - r)d}, ~
\tilde{B}_1 = e_{1} \otimes CA^{r - 1} B \in \mathbb{R}^{rd \times d}\\
& \tilde{C}_1 = 
\begin{bmatrix}
    0 & \cdots & 0  & I_{d}
\end{bmatrix} \in \mathbb{R}^{ r d}
\end{align*}
and $e_1 \in \mathbb{R}^{r}$ is the first standard basis vector.
\end{theorem}
\begin{proof}
Let us start with the form \eqref{eq: state-space system observable form}. For a system with relative degree $r \geq 1$, we have that $CA^{r - 1} B \neq 0$. As the system matrices admit a Kronecker structure, and $B = \begin{bmatrix}
    I_d \\ 0
\end{bmatrix}$ in \eqref{eq: state-space system observable form}, then the first block of $CA^{r - 1} $, i.e.,
\[
CA^{r - 1} B := g  I_d
\]
is nonsingular. 
Next, define
\begin{align*}
    Q_r = 
    \begin{bmatrix}
        C A^{r - 1} \\ \vdots \\ CA \\ C
    \end{bmatrix}
    \in \mathbb{R}^{rd \times nd}.
\end{align*}
If $\operatorname{rank} (Q_r) = rd$, we can apply an invertible linear transformation
\begin{align*}
    \tilde{\xi}_k = T \xi_k, \quad T = 
    \begin{bmatrix}
        Q_r \\ S
    \end{bmatrix}
\end{align*}
where $S = \begin{bmatrix}
    0 & S_2
\end{bmatrix} \in \mathbb{R}^{(n - r)d \times nd}$ with a full-rank matrix $S_2 \in \mathbb{R}^{(n - r) d \times ( n - 1)d}$, such that \eqref{eq: canonical structure} is derived.

If $\operatorname{rank} (Q_r) < rd$, we can apply an invertible linear transformation given by
\begin{align*}
   \bar{\xi}_{k} = \begin{bmatrix}
       C A^{r - 1}\\
       \begin{bmatrix}
           0 & I_{(n-1) d \times (n-1)d}
       \end{bmatrix}
   \end{bmatrix}
   \xi_{k}.
\end{align*}
The system is transformed to
\begin{align*}
& \begin{bmatrix}
\begin{array}{c}
y_{k+r} \\ \bar{\xi}^{(2)}_{k+1} \\ \hline y_{k}
\end{array}
\end{bmatrix}
=
    \begin{bmatrix}
    \begin{array}{cc|c}
          \bar{A}_{11} & \bar{A}_{12} &  \bar{B}_1 \\
         \bar{A}_{21} & \bar{A}_{22} & 0\\
         \hline
         \bar{C}_1 & \bar{C}_2 & 0
    \end{array}
    \end{bmatrix}
    \begin{bmatrix}
        \begin{array}{c}
        y_{k+r -1} \\ \bar{\xi}^{(2)}_{k} \\ \hline u_{k}
        \end{array}
    \end{bmatrix}
\end{align*}
where
$\bar{B}_1 = C A^{r - 1} B$,
and $\bar{\xi}^{(1)}_{k} = C A^{r- 1} \xi_k  = y_{k + r -1}$, for $r \geq 1$.
Next, we can augment the state by
\begin{align*}
\tilde{\xi}_k : = \begin{bmatrix}
   y_{k+r- 1} \\ \mathbf{y}_k \\ \bar{\xi}^{(2)}_k
\end{bmatrix}, \quad 
    \mathbf{y}_k := 
    \begin{bmatrix}
       y_{k+r - 2} \\ \vdots \\ y_{k}
    \end{bmatrix}
\end{align*}
and rewrite the output $y_k = \tilde{C}_{1} 
\begin{bmatrix}
    y_{k+r-1} \\ \mathbf{y}_k
\end{bmatrix}
$
which leads to \eqref{eq: canonical structure}.
\end{proof}
Systems in the form of \eqref{eq: canonical structure} allow a direct extension to projected algorithms, which will be shown later.

When the relative degree is one, that is, there is no delay in the input-output channel, we can obtain the following form.
\begin{corollary}[\hspace{1sp}\cite{li2025first}]
System \eqref{eq: LTI system} and \eqref{eq: gradient input to u_k} with relative degree one can be rewritten as
\begin{align}\label{eq: canonical structure relative degree one}
\begin{bmatrix}
\begin{array}{c}
y_{k+1} \\ \tilde{\xi}^{(2)}_{k+1} \\ \hline y_{k}
\end{array}
\end{bmatrix}
= &
    \begin{bmatrix}
    \begin{array}{cc|c}
         \tilde{A}_{11} & \tilde{A}_{12} & C_1 \\
         \tilde{A}_{21} & \tilde{A}_{22} & 0\\
         \hline
         I_d & 0 & 0
    \end{array}
    \end{bmatrix}
    \begin{bmatrix}
        \begin{array}{c}
        y_{k} \\ \tilde{\xi}^{(2)}_{k} \\ \hline u_{k}
        \end{array}
    \end{bmatrix}
\end{align}
where $C_1 = CB = c_1 \otimes I_d \prec 0$, as defined in \eqref{eq: state-space system observable form}.
\end{corollary}

Now that the unconstrained algorithm can be written as 
\begin{align}\label{eq: unconstrained algorithm form}
    \begin{bmatrix}
        \tilde{\xi}_{k + 1}^{(1)} \\ \tilde{\xi}_{k + 1}^{(2)}
    \end{bmatrix}
    =
    \begin{bmatrix}
        \tilde{A}_{11} & \tilde{A}_{12} \\
        \tilde{A}_{21} & \tilde{A}_{22} \\
    \end{bmatrix}
    \begin{bmatrix}
        \tilde{\xi}_{ k }^{(1)} \\ \tilde{\xi}_{ k }^{(2)}
    \end{bmatrix}
    + \tilde{B}
    \nabla f( \tilde{C} \tilde{ \xi} )
\end{align}
where $\tilde{\xi}_{k}^{(1)} = (y_{k + r - 1}, \mathbf{y}_{k} )= (y_{k + r - 1}, \ldots, y_k)$.

It has been shown in \cite{scherer2021convex} that $\tilde{A} \xi_* = \xi_*$ yields a unique solution if given a fixed $y_*$. We will show in the following lemma that more properties can hold for the system matrix $\tilde{A}$ in \eqref{eq: unconstrained algorithm form}. 
\begin{lemma}\label{lem: properties of A}
Consider $\tilde{A}$ in \eqref{eq: unconstrained algorithm form} and denote 
    \begin{align*}
    K = & \begin{bmatrix} K_1 & K_2 \end{bmatrix} := \tilde{A} - I, 
    ~ K_1 \in \mathbb{R}^{nd \times r d}, ~ 
    K_2 \in \mathbb{R}^{nd \times (n-r) d}.
    \end{align*}
    Then,
    \begin{align}\label{eq: identity condition for A}
        \left( I - K_{2} \left( K_2^{\top} K_2 \right)^{-1} K_2^\top \right) K_1 = 0
    \end{align} or equivalently,
    \begin{align}
    \begin{aligned}
    \tilde{A}_{11} - \tilde{A}_{12} \left( K_2^{\top} K_2 \right)^{-1} K_2^{\top} K_1  = I\\
    \tilde{A}_{21} - (\tilde{A}_{22} - I) \left( K_2^{\top} K_2 \right)^{-1} K_2^\top K_1 = 0.
    \end{aligned}
    \end{align}
\end{lemma}
\begin{proof}
    It has been shown that $(\tilde{A}, \tilde{C})$ is detectable without loss of generality, and $A$ has an eigenvalue at one \cite{scherer2021convex}. Then, by \cite[Theorem 3.4]{zhou1996robust}, we have the following full column rank condition,
    \[\operatorname{rank} \left( \begin{bmatrix}
        \tilde{A} - I \\ \tilde{C}
    \end{bmatrix}\right) = nd
    \]
    and thus the linear equality
    \[
        \begin{bmatrix}
            \tilde{A} - I \\ \tilde{C}
        \end{bmatrix}
        \tilde{\xi}_{*} = 
        \begin{bmatrix}
            0 \\ y_{{*}}
        \end{bmatrix}
    \]
    yields a unique solution for $\tilde{\xi}_{*}$.
    As $\tilde{C} = \begin{bmatrix}
        0_{ d \times (r - 1)d } & I_d & 0
    \end{bmatrix}$, 
    given a fixed value of $y_{*}$, we have $\xi^{(1)}_{*} = \mathbf{1}_{r} \otimes y_{*}$, and $\xi^{(2)}_{*}$ is obtained by the unique solution to 
    \begin{align}\label{eq: unique solution unconstrained case}
    \begin{aligned}
    \begin{bmatrix}
        \tilde{A}_{11} - I & \tilde{A}_{12}\\
        \tilde{A}_{21} & \tilde{A}_{22} - I
    \end{bmatrix}
    \begin{bmatrix}
        \xi^{(1)}_{*} \\ \xi^{(2)}_{*}
    \end{bmatrix}
    = K_1  \xi^{(1)}_{*} + K_2 \xi^{(2)}_{*} = 0
    \end{aligned}
    \end{align}
    which gives 
    \begin{align}\label{eq: xi^2 for unconstrained case}
    \xi^{(2)}_{*} = - K_2^{\dag} K_1 \xi^{(1)}_{*} = - \left( K_2^{\top} K_2 \right)^{-1} K_2^\top K_1 \xi^{(1)}_{*}
    \end{align}
    where $K_2^{\dag}= \left( K_2^{\top} K_2 \right)^{-1} K_2^\top$ is the pseudoinverse of $K_2$, and $K_2$ is of full column rank by \eqref{eq: unique solution unconstrained case}, that is, $\operatorname{rank} (K_2) = (n -r)d$.
    Substituting $\xi^{(2)}_{*}$ back to \eqref{eq: unique solution unconstrained case}, we obtain $\left( I - K_{2} \left( K_2^{\top} K_2 \right)^{-1} K_2^\top \right) K_1 \xi^{(1)}_* = 0$. Since $\xi^{(1)}_* = \mathbf{1}_{r} \otimes y_*$ with arbitrary $y_* \in \mathbb{R}^{d}$, we have that \eqref{lem: properties of A} holds, which completes the proof.
\end{proof}

Since the algorithm is designed to solve strongly convex optimization problems for all $f \in S(m, L)$, the first gradient input on the optimization variable must enter with a descent direction. In the canonical structure \eqref{eq: canonical structure}, this action is characterized by $\tilde{C} \tilde{B} = \tilde{B}_1 = C A^{r-1} B = g I_d$. Therefore, the scalar coefficient $g$ must satisfy $g < 0$.
\begin{lemma}\label{lem: CAB negative}
The matrix $\tilde{B}_1 = CA^{r - 1} B = g  I_d$ in system \eqref{eq: canonical structure} satisfies $g < 0$.
\end{lemma}
    

\subsection{IQC Analysis and Assumption}
Following a similar procedure as in the proof of Theorem~\ref{eq: canonical form of LTI system}, we can augment \eqref{eq: canonical structure} with the linear system
\begin{align}\label{eq: filter system}
\begin{aligned}
    \zeta_{k+1} = A_{\zeta} \zeta + B_{\zeta}^{y} y_k  + B_{\zeta}^{u} u_k,~ 
    z_k = C_{\zeta} \zeta_k + D_{\zeta}^{y} y_k  + D_{\zeta}^{u} u_k
\end{aligned}
\end{align}
such that the augmented system given by
\begin{align}\label{eq: augmented system}
\begin{aligned}
\begin{bmatrix}
        \tilde{\xi}_{k+1}\\ \zeta_{k+1} 
    \end{bmatrix}: = x_{k+1} = & \mathbf{A} x_k +  \mathbf{B} u_k\\
z_k = & \mathbf{C} x_k + \mathbf{D}  u_k\\
\mathbf{A} = &
\begin{bmatrix}
    \tilde{A} & 0\\
    B_{\zeta}^{y} \tilde{C} & A_{\zeta}
\end{bmatrix}
\in \mathbb{R}^{(n + 2 \ell ) d \times (n+ 2 \ell) d }, \\
\mathbf{B} = &
\begin{bmatrix}
    \tilde{B} \\ B_{\zeta}^{u} 
\end{bmatrix}
\in \mathbb{R}^{ (n + 2 \ell )d \times d},\\
\mathbf{C} = &
\begin{bmatrix}
    D_{\zeta}^{y} \tilde{C} & C_{\zeta}
\end{bmatrix}
\in \mathbb{R}^{  2(\ell + 1) d \times (n + 2 \ell ) d},\\
\mathbf{D} = &
    D_{\zeta}^{u} \in \mathbb{R}^{ 2(\ell + 1)d \times d}.
\end{aligned}
\end{align}
has the output $z_k$ that is a sequence of lifted signals, that is 
$z_k = (y_{k}, \ldots, y_{k - \ell}, u_{k},  \ldots, u_{k - \ell} ) \in \mathbb{R}^{2 ( \ell + 1 ) d }$.
A minimal realization of \eqref{eq: filter system} can be chosen as
\begin{align}
\begin{aligned}
    & A_{\zeta} = \operatorname{blkdiag}(A_{0}, A_{0}), \in \mathbb{R}^{ 2\ell d \times 2\ell d},~~\\
    & A_{0}= 
    \begin{bmatrix}
        0 & 0 & \cdots & 0\\
        I_{d} & 0 & \cdots & 0\\
        \vdots & \ddots & \ddots & \vdots\\
        0 & \cdots & I_{d} & 0
    \end{bmatrix} \in \mathbb{R}^{\ell d \times \ell d}, ~
    \\
    & B_{\zeta} = 
    \begin{bmatrix}
        B_{\zeta}^{y} & B_{\zeta}^{u}
    \end{bmatrix}
    = \operatorname{blkdiag} (B_0, B_0) \in \mathbb{R}^{ 2 \ell d \times 2 d},~\\
    & B_0 = 
    \begin{bmatrix}
        I_d \\ 0_{(\ell - 1) d \times d}
    \end{bmatrix} \in \mathbb{R}^{ \ell d \times d},\\
    & C_{\zeta} = \operatorname{blkdiag} (C_0, C_0), ~
    C_0 = \begin{bmatrix}
        0_{d \times \ell d} \\ I_{\ell d}
    \end{bmatrix}
   \in \mathbb{R}^{ (\ell + 1) d \times \ell d },\\
   & D_{\zeta}= \begin{bmatrix}
        D_{\zeta}^{y} & D_{\zeta}^{u}
    \end{bmatrix} = \operatorname{blkdiag} (D_0, D_0) 
    \in \mathbb{R}^{2(\ell + 1 ) d \times 2 d}, ~\\
    & D_0 = 
    \begin{bmatrix}
        I_d \\ 0_{\ell d \times d} 
    \end{bmatrix}
    \in \mathbb{R}^{( \ell + 1 ) d \times d}.
\end{aligned}
\end{align}
As the inputs are given by $u_k  = \nabla f (y_k)$, for all $k \geq 0$, pointwise conditions (in terms of $z_k$) can be derived for the sequence of input/output pairs that share the repeated nonlinearities $\nabla f (\cdot)$, that is,
\begin{align}
    \left(z_k - z_*\right)^{\top}
    \mathbf{M}
    \left(z_k - z_*\right)
    \geq 0, ~\text{for all } k \geq 0
\end{align}
where $z_* = (y_*, \ldots, y_*, u_*,\dots, u_*)$ is a reference point, generating a stepwise decreasing Lyapunov function for stability analysis \cite{kulkarni2002all,van2022absolute,carrasco2019convex,li2025first}.
Define
\begin{align*}
& T_\rho := \operatorname{diag}(1,\rho,\rho^2,\dots,\rho^\ell)\in\mathbb{R}^{ (\ell + 1)\times (\ell + 1)}, \\
& \mathcal{T}_{\ell,m,L} : = \begin{bmatrix}
    L & -1\\ -m & 1
\end{bmatrix} \otimes I_{\ell + 1} \in \mathbb{R}^{2(\ell+1) \times 2(\ell+1) }
\end{align*}
The multiplier is chosen as
\begin{subequations}\label{eq: multiplier M}
\begin{align}
\mathbf{M}
&=
{\mathbf{M}}_{Q}
+
{\mathbf{M}}_{\tilde{Q}},
\\
{\mathbf{M}}_{Q}
&=
\left( \mathcal{T}_{\ell,m,L}^\top 
\begin{bmatrix}
0 & Q^{\top}\\
Q & 0
\end{bmatrix}
\mathcal{T}_{\ell,m,L} \right) 
\otimes I_d,
\\
{\mathbf{M}}_{\tilde{Q}}
&=
\left( 
\mathcal{T}_{\ell,m,L}^\top 
\begin{bmatrix}
0 & (T_\rho \tilde{Q} T_\rho)^{\top}\\
T_\rho \tilde{Q} T_\rho & 0
\end{bmatrix}
\mathcal{T}_{\ell,m,L} 
\right)
\otimes I_d,
\end{align}
\end{subequations}
where matrices $Q$ and $\tilde{Q}$ are constrained to be doubly hyperdominant, namely, $Q, \, \tilde{Q} \in \mathcal{Q}$, with
\begin{align*}
\mathcal{Q}= \left\{ Q \in \mathbb{R}^{(\ell + 1) \times (\ell + 1) }: Q_{ij}\le 0,~\forall\, i\neq j, ~ \right.\\
\quad \left. \phantom{Q \in \mathbb{R}^{(\ell + 1) \times (\ell + 1) }}  Q\mathbf{1}\ge 0, ~ Q^{\top}\mathbf{1}\ge 0 \right\}.
\end{align*}
Here, \(\mathcal{T}_{\ell,m,L}\in\mathbb{R}^{2(\ell+1)d\times 2(\ell+1)d}\) denotes the lifted loop transformation from the slope-restricted $[m, L]$ to monotone in $[0, \infty]$, see, e.g., \cite{fetzer2017absolute}. The matrix $\mathbf{M}_{Q}$ represents the IQC for the repeated monotone nonlinearities \cite{van2022absolute,boczar2017exponential}, and $\mathbf{M}_{\tilde{Q}}$ represents the the $\rho$-hard IQC multiplier for $\tilde{z}_{k} = (\tilde{y}_{k}, \ldots, \tilde{y}_{k - \ell}, \tilde{u}_{k}, \ldots, \tilde{u}_{k - \ell})$, where $\tilde{y}_k = \rho^{-k} y_k$, $\tilde{u}_k = \rho^{-k} u_k$, but starting from time $k$, that is 
\begin{align*}
\sum_{t = k}^{T} \rho^{-2 t}
    \left(z_t - z_*\right)^{\top}
    \mathbf{M}
    \left(z_t - z_*\right)
    \geq 0, ~\text{for all } T > t \geq 0,
\end{align*}
see, e.g., \cite{lessard2016analysis,boczar2017exponential}.
These conditions ensure that $\mathbf{M}$ defines a valid pointwise quadratic constraint for the repeated monotone nonlinearity  \cite{kulkarni2002all,van2022absolute}.
The convergence of the unconstrained algorithms \eqref{eq: gradient input to u_k}, \eqref{eq: canonical structure} is characterized by the following theorem, whose proof is similar to \cite{lessard2016analysis} and omitted here.
\begin{theorem}\label{thm: IQC analysis for unconstrained algorithm}
Let $(\mathbf{A}, \mathbf{B},\mathbf{C},\mathbf{D})$ be defined by \eqref{eq: augmented system} and $M$ given in \eqref{eq: multiplier M}.
Suppose there exist a $\mathbf{P} \succ 0$ and $\rho \in (0, 1)$, such that the following LMI,
\begin{align}\label{eq: LMI from IQC}
    \begin{bmatrix}
        \mathbf{A}^{\top} \mathbf{P} \mathbf{A} - \rho^2 \mathbf{P} & \mathbf{A}^{\top} \mathbf{P} \mathbf{B}\\
        \mathbf{B}^{\top} \mathbf{P} \mathbf{A} & \mathbf{B}^{\top} \mathbf{P} \mathbf{B}
    \end{bmatrix}
    +
    \begin{bmatrix}
        \mathbf{C}^{\top} \\ \mathbf{D}^{\top}
    \end{bmatrix}
    \mathbf{M}
    \begin{bmatrix}
        \mathbf{C} & \mathbf{D}
    \end{bmatrix}
    \preceq 0
\end{align}
is feasible. Then, the unconstrained optimization algorithm \eqref{eq: gradient input to u_k}, \eqref{eq: canonical structure} is linearly convergent with rate $\rho$, and there exists a quadratic positive definite Lyapunov function $V (x_{k}) = \left(x_{k} - x_{*}\right)^\top \mathbf{P} \left(x_{k} - x_{*}\right)$, such that for all $k \geq 1$,
\[V (x_{k + 1}) = \left(x_{k + 1} - x_{* + 1}\right)^\top \mathbf{P} \left(x_{k + 1} - x_{* + 1}\right) \leq \rho^2 V (x_{k}),\]
where $(x_*, u_*, z_*)$ with $z_* = (y_*, \ldots, y_*, u_*,\dots, u_*)$, is a reference point of \eqref{eq: augmented system}, and $x_{* + 1}$ denotes the next state of $x_*$.
\end{theorem}

We assume the feasibility of LMI \eqref{eq: LMI from IQC}, which is summarized by the following assumption.
\begin{assumption}\label{assumption 1}
    For the unconstrained algorithm \eqref{eq: gradient input to u_k}, \eqref{eq: canonical structure}, there exists a lifting integer $\ell \in \mathbb{N}$ such that Theorem~\ref{thm: IQC analysis for unconstrained algorithm} holds.
    %
    %
\end{assumption}
This assumption is aligned with a time-domain hard IQC imposed on any truncated horizon with arbitrary initial time \cite{van2022absolute}, \cite{li2025first}, which generally holds for many multipliers \cite{lessard2016analysis}.

\section{Projected Algorithm and Analysis}\label{sec: Projected Algorithm and Analysis}
\subsection{Projected Algorithm on the Lyapunov-induced Norm}
Now, for the constrained optimization problem \eqref{eq: constrained convex optimization problem}, we propose the projected algorithm that preserves the same convergence rates.

The most straightforward projected method is a naive direct projection of the $y_{k + r }$ state back to the constraint set, that is,
\begin{subequations}\label{eq: state-space form of projected algorithm}
\begin{align}
\begin{bmatrix}
y_{k+ r - \frac{1}{2} } \\ y_{k + r - 1 } \\ \vdots \\ y_{k + 1} \\ \tilde{\xi}^{(2)}_{k+1}
\end{bmatrix}
= & \tilde{A} 
    \begin{bmatrix}
    y_{k + r -1 }  \\ y_{k + r - 2 } \\ \vdots \\ y_{k} \\ \tilde{\xi}^{(2)}_{k} 
    \end{bmatrix} 
    + \tilde{B} \nabla f( y_k )\\
y_{k + r} = & \Pi_{\Omega} \left(y_{k + r - \frac{1}{2}} \right)
\end{align}
\end{subequations}
where $y_{k + r - \frac{1}{2}}$ represents an intermediate state that will be projected back to the constraint set $\Omega$, $(\tilde{A}, \tilde{B})$ are given in \eqref{eq: canonical structure}, and $\Pi_{\Omega} (\cdot)$ is the projection onto set $\Omega$.
It can be analyzed similarly to \cite{li2025first} that the unique fixed point of \eqref{eq: state-space form of projected algorithm} provides the optimal solution to problem \eqref{eq: constrained convex optimization problem}.
However, it is also clear that \eqref{eq: state-space form of projected algorithm} is a projection on the Euclidean norm, under which the error bound is not guaranteed to decrease monotonically at the same rate.

Instead, we start with the augmented system in \eqref{eq: augmented system}. Let us partition the state as $x_k = (y_{k + r - 1}, \mathbf{y}_{k}, \eta_k)$, then $\mathbf{A}$, $\mathbf{B}$ and the Lyapunov matrix $\mathbf{P}$ are partitioned as
\begin{align*}
    & \mathbf{A} = \begin{bmatrix}
        \mathbf{A}_{11} & \mathbf{A}_{1r} & \mathbf{A}_{1r_{+}} \\
        \mathbf{A}_{r1} & \mathbf{A}_{rr} & \mathbf{A}_{rr_{+}}\\
        \mathbf{A}_{r_{+} 1 } & \mathbf{A}_{r_{+} r} & \mathbf{A}_{r_{+}r_{+}}
    \end{bmatrix}, ~ \mathbf{A}_{11} \in \mathbb{R}^{d \times d},\\
    & \mathbf{B} = \begin{bmatrix}
        \mathbf{B}_{1} \\ \mathbf{B}_{r} \\ \mathbf{B}_{r_{+}}
    \end{bmatrix},
    ~\mathbf{B}_{1} = g I_{d}, ~ g < 0,\\
    & \mathbf{P} = 
    \begin{bmatrix}
        \mathbf{P}_{11} & \mathbf{P}_{1r} & \mathbf{P}_{1r_{+}} \\
        \mathbf{P}_{r1} & \mathbf{P}_{rr} & \mathbf{P}_{rr_{+}}\\
        \mathbf{P}_{r_{+} 1 } & \mathbf{P}_{r_{+} r} & \mathbf{P}_{r_{+}r_{+}}
    \end{bmatrix}, \mathbf{P}_{11} \in \mathbb{R}^{d \times d}.
\end{align*}
respectively, where $g < 0$ by Lemma~\ref{lem: CAB negative}. 

We propose the projected algorithm on the Lyapunov-induced norm:
\begin{align}\label{alg}
\begin{aligned}
    x_{k + \frac{1}{2} } = & \mathbf{A} x_k + \mathbf{B}\nabla f ( y_{k} ), ~~ 
    y_{k} = 
    \begin{bmatrix}
        \tilde{C} & 0
    \end{bmatrix} x_k
    \\
    x_{k + 1} = & \Pi_{\mathcal{D}}^{\mathbf{P}} (x_{k+ \frac{1}{2} })
\end{aligned}
\end{align}
where $(\mathbf{A}, \mathbf{B})$ are given in \eqref{eq: augmented system}, $\tilde{C}$ is defined in \eqref{eq: canonical structure}, and $\Pi_{\mathcal{D}}^{\mathbf{P}} ( \cdot )$ is the projection onto set $\mathcal{D}$ in the $\mathbf{P}$ weighted norm, with
\begin{align}
\begin{aligned}
& \mathcal{D} := \mathcal{Y} \times \mathbb{R}^{(n + 2 \ell - r) d}\\
& \mathcal{Y} := \left\{ \xi \in \mathbb{R}^{rd} : \xi^{(1)} \in \Omega, ~ \mathbf{A}_{r1} \xi^{(1)} = (I -  \mathbf{A}_{rr} )\xi^{(2)}\right\}.
\end{aligned}
\end{align}
Here,  $\mathcal{Y}$ includes the physical hard constraint enforcing the next step of $y_{t}$ should be $y_{t+1}$, for $t = k, k+1, \ldots, k+r-1$. 

By solving the $\mathbf{P}$-norm projection in \eqref{alg}, we can obtain the explicit update:
\begin{subequations}
\begin{align}
y_{k+ r - \frac{1}{2} } = & \mathbf{A}_{11} y_{k + r  - 1} + \mathbf{A}_{1r} \mathbf{y}_{k} + \mathbf{A}_{1r_{+}} \eta_{k} + g \nabla f ( y_k ) \\
y_{k + r} = & \Pi_{\Omega}^{\mathbf{S}} ( y _{k+ r - \frac{1}{2}} ) \label{eq: projection on the S norm}\\
\eta_{k + 1} = & \mathbf{A}_{r_{+}1} y_{k + r - 1} + \mathbf{A}_{r_{+} r} \mathbf{y}_{k} + \mathbf{A}_{r_{+}r_{+}} \eta_{k} \nonumber \\
& + B_{\zeta}^{u} \nabla f(y_k) - \mathbf{P}_{r_{+} r_{+}}^{-1} \mathbf{P}_{1r_{+}}^{\top} \left(y_{k + r} - y_{k+ r - \frac{1}{2} }\right)
\end{align}
\end{subequations}
where $\mathbf{S} := \mathbf{P}_{11} - \mathbf{P}_{1r_{+}} \mathbf{P}_{r_{+}r_{+}}^{-1} \mathbf{P}_{1r_{+}}^{\top} \succ 0$. Because $\mathbf{S} \in \mathbb{R}^{d}$ also admits a Kronecker structre, the projection in \eqref{eq: projection on the S norm} is equivalent to a standard projection $\Pi_{\Omega} (\cdot)$. 
Also, the update of $\mathbf{y}_{k}$ is trivially $\mathbf{y}_{k+1} = (y_{k+r - 1}, \ldots, y_{k+1})$.

As the algorithmic state does not depend on the IQC filter state except for the information from the Lyapunov matrix $\mathbf{P}$, we only need to implement algorithmic state $\xi_k = (y_{k+r-1}, \mathbf{y}_{k}, \xi^{(2)}_k)$ in practice, that is,
\begin{subequations}\label{eq: projected algorithm clean}
\begin{align}
    \xi^{(1)}_{k + \frac12 } = & \begin{bmatrix}
        y_{k + r - \frac12 } \\ \mathbf{y}_{k + \frac12 }
    \end{bmatrix}
    = \tilde{A}_{11} \xi^{(1)}_{k} + \tilde{A}_{12} \xi^{(2)}_{k} + \tilde{B}_{1} \nabla f \left( \tilde{C}_{1} \xi^{(1)}_k \right) \\
 \xi^{(1)}_{k + 1 }= & 
 \begin{bmatrix}
     \Pi_{\Omega} \left( y _{k+ r - \frac{1}{2}} \right) \\ \mathbf{y}_{k+\frac12 } 
 \end{bmatrix}\\
\xi^{(2)}_{k + 1} = & \tilde{A}_{21} \xi^{(1)}_{k} + \tilde{A}_{22} \xi^{(2)}_{k} + \chi \left(y_{k + r} - y_{k+ r - \frac{1}{2} } \right)
\end{align}
\end{subequations}
where system matrices are given in \eqref{eq: canonical form of LTI system}, and \[\chi: = \begin{bmatrix}
    I_{(n - r) d} & 0 
\end{bmatrix} \mathbf{P}_{r_{+} r_{+}}^{-1} \mathbf{P}_{1r_{+}}^{\top} \in \mathbb{R}^{(n - r) d \times d} \] is the first $(n - r)d$ rows of $\mathbf{P}_{r_{+} r_{+}}^{-1} \mathbf{P}_{1r_{+}}^{\top}$.

\subsection{Optimality and Rate-preserving Convergence}
We first discuss the fixed point of the projected algorithm \eqref{alg}, summarized by the following theorem.
\begin{theorem}\label{thm: fixed point gives optimum}
Assume that for any $f \in S (m, L)$ and closed convex set $\Omega$, Algorithm~\eqref{alg} admits a unique fixed point
\( x_* = (y_*, \mathbf{y}_{*}, \eta_*) \).
Then, $y_*$ is the unique optimal solution to problem
\eqref{eq: constrained convex optimization problem}.
\end{theorem}

\begin{proof}
Let 
\[
x_{*+\frac12} =
\begin{bmatrix}
\xi_{*+\frac12}\\ \eta_{* + \frac12} 
\end{bmatrix}
:=
\mathbf A x_*+\mathbf B \nabla f(y_*).
\]
Since $x_*$ is a fixed point of Algorithm~\eqref{alg}, we also have
\[
x_*=\Pi_{\mathcal D}^{\mathbf P}(x_{*+\frac12}).
\]
As the algorithmic state $\xi_k$ does not depend on the IQC
filter state $\zeta_k$, it is sufficient to analyze the fixed point of \eqref{eq: projected algorithm clean} solely. 
Expanding the above equations, we have 
\begin{subequations}\label{eq: fp_proj_conditions}
\begin{align}
\xi^{(1)}_{* + \frac12 } = & \begin{bmatrix}
        y_{* + \frac12 } \\ \mathbf{y}_{* + \frac12 }
    \end{bmatrix}
    = \tilde{A}_{11} \xi^{(1)}_{*} + \tilde{A}_{12} \xi^{(2)}_{*} + \begin{bmatrix}
        g \nabla f \left( y_{*} \right) \\ 0
    \end{bmatrix} \\
 y_{*}= & \Pi_{\Omega} \left( y _{* + \frac12 } \right) \\
\xi^{(2)}_{*} =&  \tilde{A}_{21} \xi^{(1)}_* + \tilde{A}_{22} \xi^{(2)}_*
-\chi \left( y_*-y_{*+\frac12} \right), \label{eq: fp_proj_conditions_c}
\end{align}
\end{subequations}
and
\begin{align}\label{eq: normal_cone_from_projection}
y_{*+\frac12}-y_* \in N_{\Omega}(y_*).
\end{align}
Let $[K_1 ~ K_2]: = \tilde{A} - I$, we obtain by rearranging \eqref{eq: fp_proj_conditions},
\begin{align}\label{eq: fixed point condition compact}
K_1 \xi^{(1)}_{*}
+ K_2 \xi^{(2)}_{*} = - \tilde{B} \nabla f (y_*) - H r_*
\end{align}
where $H = \begin{bmatrix}
    I_{d} \\ 0 \\ \chi
\end{bmatrix}$ and $r_* : = y_* - y_{*+ \frac12}$.
By Lemma~\ref{lem: properties of A}, left multiplying $N: = I - K_2 K_2^{\dag}$ to both side of \eqref{eq: fixed point condition compact} gives
\begin{align*}
    N \left( \tilde{B} \nabla f (y_*) +  H r_* \right) = 0
\end{align*}
where $N K_2 = 0$ follows from the fact that $N$ is the orthogonal projection onto $ \left(\text{im} (K_2) \right)^{\perp}$.
%
Since $\tilde{B} = E g$, with $E = \begin{bmatrix}
    I_{d} \\ 0 
\end{bmatrix}$, 
left multiplying $E^\top$ to the above equation, we obtain
\begin{align*}
    (E^\top N E) g \nabla f(y_*) = - E^\top N  H r_*.
\end{align*}
Note that $E^\top N E = E^\top N^\top N E \succ 0$, as $N E \in \mathbb{R}^{d}$ is nonsingular in a well-posed unconstrained algorithmic structure, that is, when $r_* = 0$.
Therefore, by \eqref{eq: normal_cone_from_projection}, we have 
\begin{align}\label{eq: fixed point final condition}
    {g}\nabla f(y_*) \in \Gamma N_{\Omega}(y_*)
\end{align}
where \[
\Gamma :=
(E^\top N E)^{-1} E^\top N  H =
\gamma I_d
\]
since the above matrices admit a Kronecker structure.
As $g <0$ by Lemma~\ref{lem: CAB negative}, then \eqref{eq: fixed point final condition} represents a necessary condition that the optimal solution to problem $\min_{y \in \Omega} \gamma f(y)$ should satisfy.
By the assumption of the uniqueness of $y_*$, $\gamma$ must be positive.
\end{proof}

\begin{theorem}\label{thm: same_rate}
    Suppose Assumption~\ref{assumption 1} holds \ml{such that the unconstrained algorithm \eqref{eq: gradient input to u_k}, \eqref{eq: canonical structure} with relative degree $r \geq 1$ linearly converges with rate $\rho \in (0, 1)$}, then the corresponding projected algorithm~\eqref{alg} also linearly converges at the same rate $\rho$. 
\end{theorem}

\begin{proof}
Given a reference point $(x_*, u_*, z_*)$, where $u_* = \nabla f( \tilde{C} x_*) = \nabla f (y_*)$, and $z_* = (x_*, \ldots, x_*, u_*, \ldots u_*)$,
we have 
\[x_{* + \frac12} = \mathbf{A} x_* + \mathbf{B} u_*, \quad z_* = \mathbf{C} x_* + \mathbf{D} u_*.\]
Then, by Assumption~\ref{assumption 1}, the Lyapunov inequality for the unconstrained half-step dynamics gives
\[
\| x_{k + \frac12 }  - x_{*+\frac12} \|_{\mathbf{P}}^{2} 
\leq
\rho^2 \| x_{k }  - x_{*} \|_{\mathbf{P}}^{2}.
\]
By the non-expansiveness of the projection on the $\mathbf{P}$ norm,
\begin{align*}
    & \| x_{k+1}-  x_*\|_{\mathbf{P}}^{2}
    =
\|\Pi_{\mathcal{D}}^{\mathbf{P}}( x_{k+\frac12} ) -\Pi_{\mathcal{D}}^{\mathbf{P}}( x_{*+\frac12})\|_{\mathbf{P}}^{2}\\
\le &
\| x_{k+\frac12}- x_{*+\frac12}\|_{\mathbf{P}}^{2}
\leq \rho^2 \| x_{k}- x_{*}\|_{\mathbf{P}}^{2}\\
\leq & \ldots \leq \rho^{2( k + 1 ) } \| x_{0}- x_{*}\|_{\mathbf{P}}^{2}
\end{align*}
Consequently, 
\[
\|x_k-x_*\|_2
\le
\sqrt{\operatorname{cond} (\mathbf{P})}\,
\rho^k
\|x_0-x_*\|_2
\]
where $\operatorname{cond} (\mathbf{P})$ is the condition number of $\mathbf{P}$.
Therefore, Algorithm~\eqref{alg} converges linearly to $x_*$ with rate $\rho$. Since this linear convergence is global, we have that the fixed point pair $(x_*, x_{*+\frac12} )$ to the projected dynamics $x_{*} = \Pi_{\mathcal{D}}^{\mathbf{P}} \left(x_{*+\frac12} \right)$, $x_{* + \frac12} = \mathbf{A} x_* + \mathbf{B} \nabla f (y_*)$, is unique, and its partial state $y_*$ is the optimal solution to the constrained problem \eqref{eq: constrained convex optimization problem} by Theorem~\ref{thm: fixed point gives optimum}.
\end{proof}


\section{Numerical Example}\label{Numerical Example}
\subsection{Projected Algorithm with One-step Delay}
Consider the problem of solving $\min_{y} f(y)$ in a gradient feedback channel with one-step delay.
It has been shown in \cite{scherer2023optimization,scherer2025tutorial} that the gradient descent method $x_{k+1} = x_k - \alpha \nabla f  (x_{k -1 })$ fails to converge under delays.
Let us set $L = 10$, $m = 1$, and use the software developed by
\cite{Scherer2023AlgorithmSynthesis} to synthesize a first-order algorithm, given by
\begin{align*}
G (z) = 
\frac{-0.152 (z - 0.342) (z+8.749 \times 10^{-5})}{z (z-0.580) (z-1) (z+0.896) }.
\end{align*}
The above can be rewritten into the canonical state-space form with the state $x_k = (y_{k + 1}, y_{k}, \xi^{(1)}_{k}, \xi^{(2)}_{k})$: 
\begin{align}\label{eq: unconstrained alg example}
\begin{aligned}
\begin{bmatrix}
    y_{k + 2} \\ y_{k+1} \\ \xi^{(1)}_{k+1} \\ \xi^{(2)}_{k+1} 
\end{bmatrix}
= &
\begin{bmatrix}
    0.342 &  2.297  & 0.204  &  -0.157\\
    1 & 0 & 0 & 0\\
    -6.583 & -17.788 & -2.044 &  1.571\\
   0  & -24.838  &  -3.104  &  2.386
\end{bmatrix}
\begin{bmatrix}
    y_{k + 1} \\ y_{k} \\ \xi^{(1)}_{k} \\ \xi^{(2)}_{k} 
\end{bmatrix}\\
& + 
\begin{bmatrix}
    -0.1519 \\ 0 \\ 0 \\ 0
\end{bmatrix}
u_{k}\\
:= & 
\begin{bmatrix}
    A_{11} & A_{12} & A_{13} \\
    A_{21} & A_{22} & A_{23} \\
    A_{31} & A_{32} & A_{33}
\end{bmatrix}
\begin{bmatrix}
    y_{k+1} \\ {y}_k \\ \xi_k
\end{bmatrix}
 + \begin{bmatrix}
     B_{1} \\ 0 \\ 0
 \end{bmatrix} u_k\\
y_{k} = & 
\begin{bmatrix}
    0  & 1    &   0  &    0
\end{bmatrix}
x_k
:= C x_k, ~
u_{k} = \nabla f(y_k).
\end{aligned}
\end{align}
Here, $C A^{r- 1} B = B_1 = -0.152 < 0$, since the relative degree $r = 2$.
We solve the LMI on MATLAB using YALMIP with solver MOSEK. With a lifted dimension $\ell = 9$, we carry out a bisection search and obtain the rate $\rho = 0.827$, which is the same as \cite{scherer2025tutorial,Scherer2023AlgorithmSynthesis}.
The first two elements of $\mathbf{P}_{33}^{-1} \mathbf{P}_{13}^{\top}$ are given by 
\[\chi = 
\begin{bmatrix}
     -0.029  \\ 0.036
\end{bmatrix}.
\]
Then the projected algorithm for problem \eqref{eq: constrained convex optimization problem} is constructed as
\begin{subequations}\label{eq:alg example}
\begin{align}
 & y_{k + 1 + \frac12} =  A_{11} y_{k+1} + A_{12} {y}_{k} + A_{13} \xi_{k}+ B_1 \nabla f(C x_k) \\
 & y_{k + 2} = \Pi_{\Omega}  \left(  y_{k + 1 + \frac12} \right)\\
& \xi_{k+1} = A_{31} y_{k+1} + A_{32} y_{k} + A_{33} \xi_{k} -
\chi
\left( y_{k + 2} - y_{k + 1 + \frac12}  \right)
\end{align}
\end{subequations}
By Theorem~\ref{thm: same_rate}, the projected \eqref{eq:alg example} converges linearly to the optimal solution of the constrained problem \eqref{eq: constrained convex optimization problem} with the same rate $\rho = 0.827$, using the same parameters of \eqref{eq: unconstrained alg example}.
\subsection{Numerical Implementation}
Let us consider a numerical example with $y \in \mathbb{R}^{2}$, and
\begin{align*}
    f (y) = \frac12 y^\top F y + p^\top y, ~
    F = \begin{bmatrix}
        9.88 & -1\\ -1 & 1.117
    \end{bmatrix},~
    p = \begin{bmatrix}
        1 \\ 5
    \end{bmatrix}.
\end{align*}
It can be verified that $f \in S (1.0, 9.99)$.
The trajectory of the convergence error $\| y_k  - y_*\|_2$ for \eqref{eq: unconstrained alg example} with random initial condition $x_0 \in [0,1]^{8}$ is depicted in Fig~\ref{fig:1}.
\begin{figure}[htbp]
    \centering
    \includegraphics[width=1\linewidth]{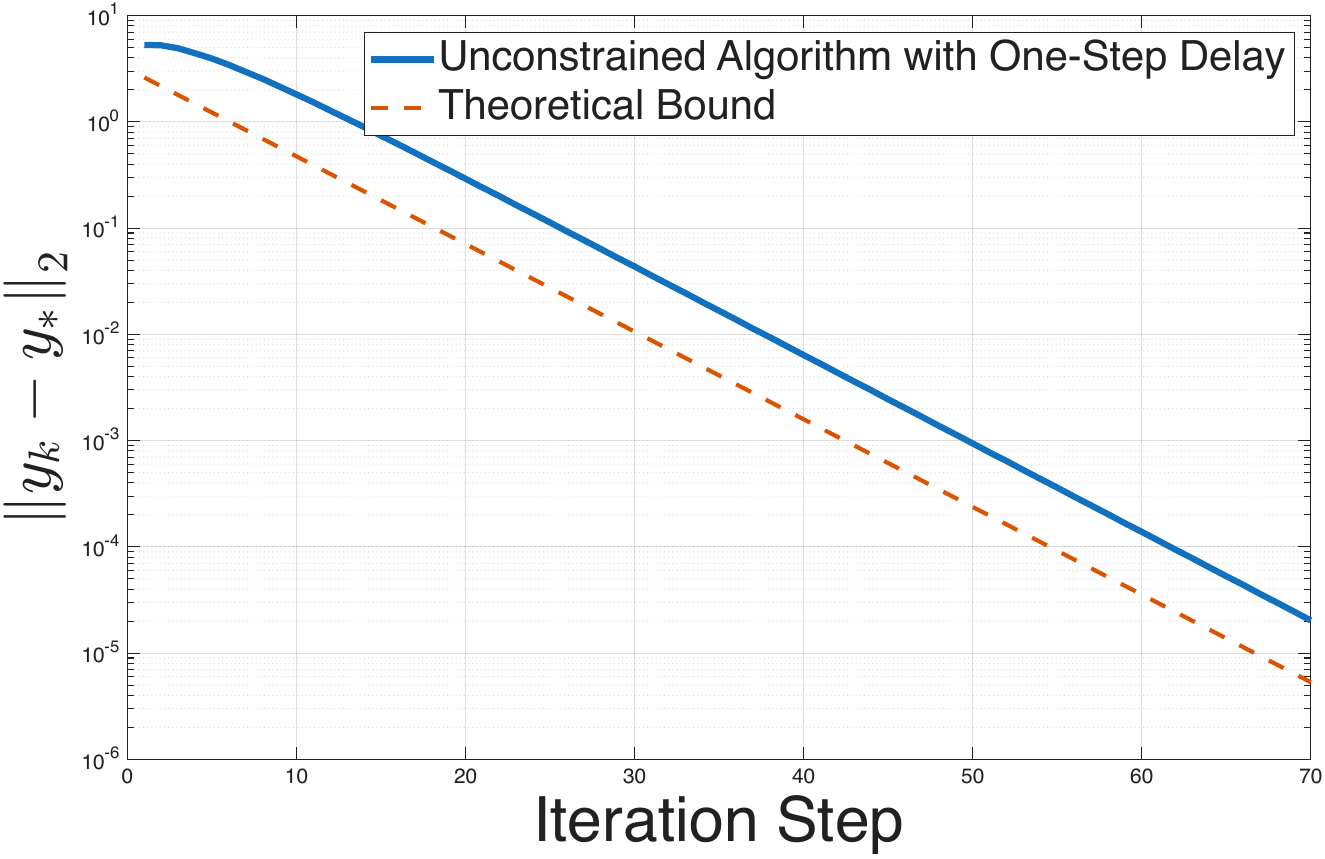}
    \caption{The trajectories of $\| y_k  - y_*\|_2$ for the unconstrained algorithm with one-step delay in \eqref{eq: unconstrained alg example}, where $y_*$ is the optimal solution to the unconstrained problem.}
    \label{fig:1}
\end{figure}
Next, we adopt an ellipse constraint set
\begin{align*}
    \Omega = \left\{ y \in \mathbb{R}^{2}: y^\top \begin{bmatrix}
        1 & -0.5\\ -0.5 & 2
    \end{bmatrix} y \leq 10 \right\}.
\end{align*}
The trajectories of the convergence error $\| y_k  - y_*\|_2$ and $|f(y_k)  - f(y_*)|$ for the projected algorithm \eqref{eq:alg example} under one-step delay feedback with random initial condition $x_0 \in [0,1]^{8}$ are depicted in Fig~\ref{fig:21} and Fig~\ref{fig:22}.
\begin{figure}[htbp]
    \centering
    \begin{subfigure}[t]{1\linewidth}
    \centering
    \includegraphics[width=1\linewidth]{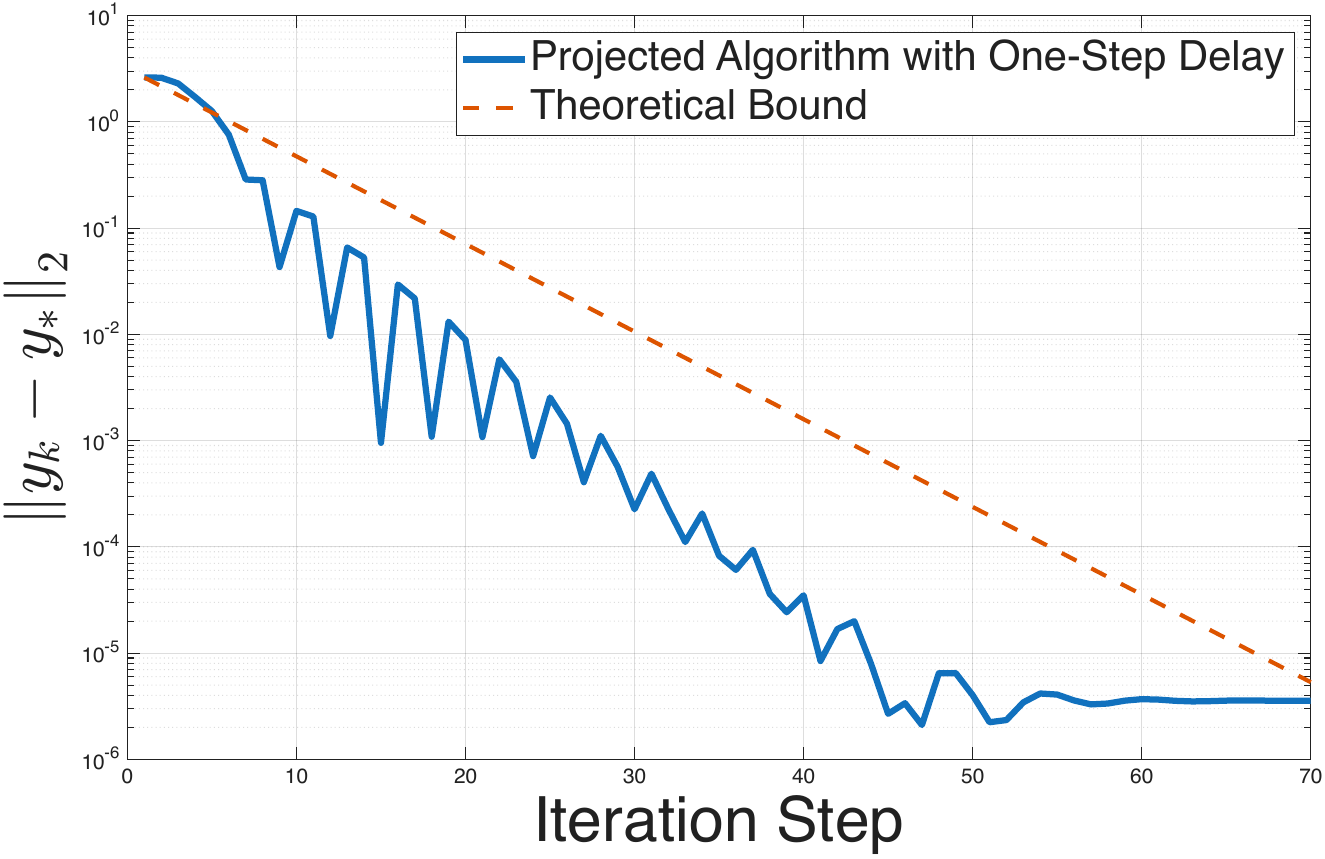}
    \caption{Trajectory of the distance $\| y_{k} - y_* \|_2$, where $y_*$ is the optimal solution to the constrained problem.}
    \label{fig:21}
\end{subfigure}
\begin{subfigure}[t]{1\linewidth}
    \centering
    \includegraphics[width=1\linewidth]{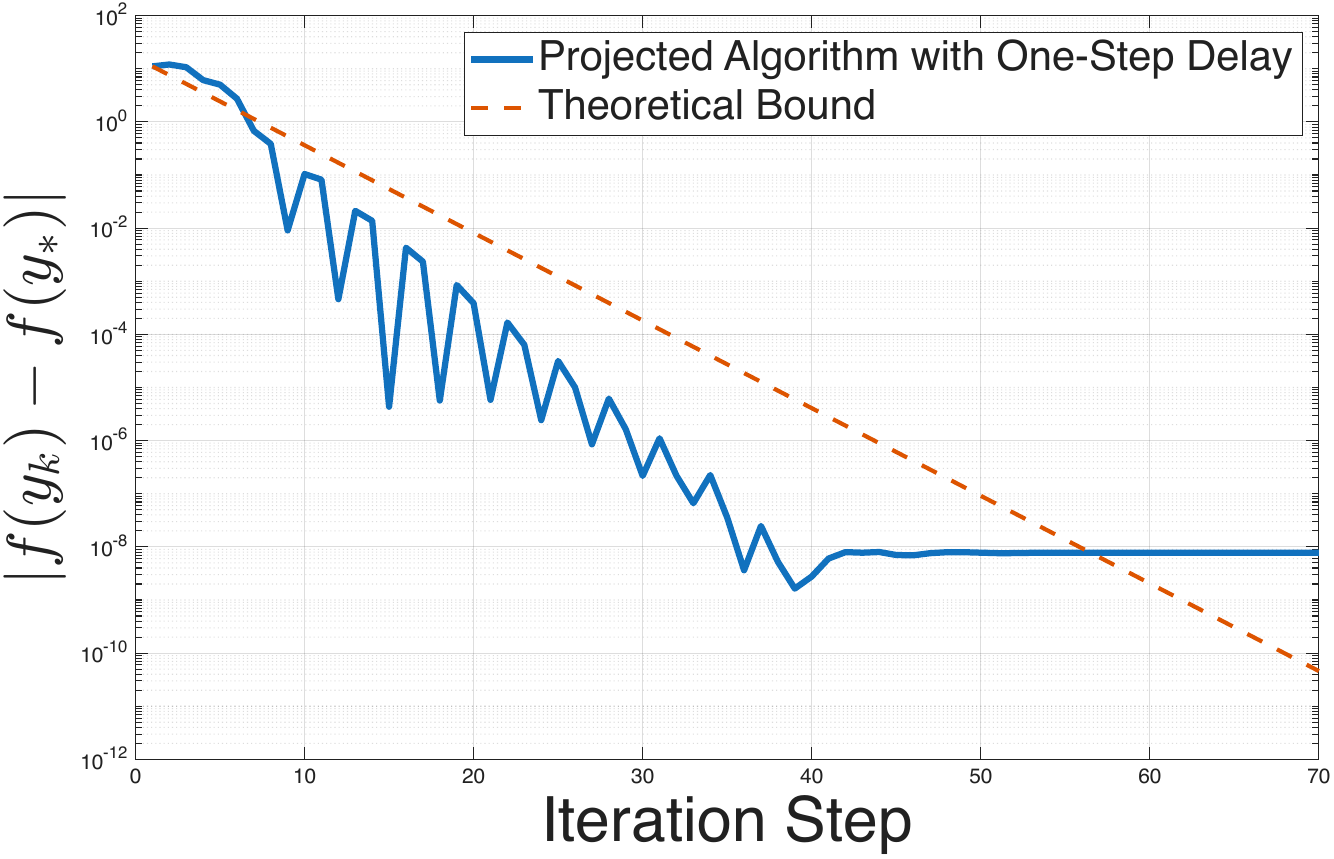}
    \caption{Trajectory of the error $| f (y_k) - f(y_*) |$.}
    \label{fig:22}
\end{subfigure}
\caption{Convergence error for the proposed projected algorithm with one-step delay in \eqref{eq:alg example}. The theoretical rate bound is given by the dotted lines.}
\label{fig:2}
\end{figure}

\section{Conclusion}\label{Conclusion}
We have proposed a canonical framework for extending first-order unconstrained algorithms in Lur’e form, with arbitrary relative degree, to projected algorithms that preserve the same convergence rate. Future work includes extending the framework to proximal algorithms. Another important direction is to investigate the sufficiency and necessity of the class of multipliers used in our method.

\section*{Acknowledgment}
M.~Li used ChatGPT to assist with coding and manuscript polishing.

\bibliographystyle{IEEEtran}
\bibliography{References}

@book{zhou1996robust,
  title={Robust and optimal control},
  author={Zhou, Kemin and Doyle, John Comstock and Glover, Keith},
  volume={40},
  year={1996},
  publisher={New Jersey: Prentice Hall}
}

@article{hu2016exponential,
  title={Exponential decay rate conditions for uncertain linear systems using integral quadratic constraints},
  author={Hu, Bin and Seiler, Peter},
  journal={IEEE Transactions on Automatic Control},
  volume={61},
  number={11},
  pages={3631--3637},
  year={2016},
  publisher={IEEE}
}

@article{megretski1997system,
  title={System analysis via integral quadratic constraints},
  author={Megretski, Alexandre and Rantzer, Anders},
  journal={IEEE Transactions on Automatic Control},
  volume={42},
  number={6},
  pages={819--830},
  year={1997},
  publisher={IEEE}
}

@book{nesterov2003introductory,
  title={Introductory lectures on convex optimization: A basic course},
  author={Nesterov, Yurii},
  volume={87},
  year={2003},
  publisher={Springer Science \& Business Media}
}

@article{li2025convergence,
  title={Convergence rate bounds for the mirror descent method: {IQCs}, {Popov} criterion and {Bregman} divergence},
  author={Li, Mengmou and Laib, Khaled and Hatanaka, Takeshi and Lestas, Ioannis},
  journal={Automatica},
  volume={171},
  pages={111973},
  year={2025},
  publisher={Elsevier}
}

@article{li2020smooth,
  title={Smooth dynamics for distributed constrained optimization with heterogeneous delays},
  author={Li, Mengmou and Yamashita, Shunya and Hatanaka, Takeshi and Chesi, Graziano},
  journal={IEEE Control Systems Letters},
  volume={4},
  number={3},
  pages={626--631},
  year={2020},
  publisher={IEEE}
}

@inproceedings{zheng2017asynchronous,
  title={Asynchronous stochastic gradient descent with delay compensation},
  author={Zheng, Shuxin and Meng, Qi and Wang, Taifeng and Chen, Wei and Yu, Nenghai and Ma, Zhi-Ming and Liu, Tie-Yan},
  booktitle={International Conference on Machine Learning},
  pages={4120--4129},
  year={2017},
  organization={PMLR}
}

@article{hatanaka2018passivity,
  title={Passivity-based distributed optimization with communication delays using PI consensus algorithm},
  author={Hatanaka, Takeshi and Chopra, Nikhil and Ishizaki, Takayuki and Li, Na},
  journal={IEEE Transactions on Automatic Control},
  volume={63},
  number={12},
  pages={4421--4428},
  year={2018},
  publisher={IEEE}
}

@article{fetzer2017absolute,
  title={Absolute stability analysis of discrete time feedback interconnections},
  author={Fetzer, Matthias and Scherer, Carsten W},
  journal={IFAC-PapersOnLine},
  volume={50},
  number={1},
  pages={8447--8453},
  year={2017},
  publisher={Elsevier}
}

@inproceedings{scherer2023optimization,
  title={Optimization algorithm synthesis based on integral quadratic constraints: A tutorial},
  author={Scherer, Carsten W and Ebenbauer, Christian and Holicki, Tobias},
  booktitle={2023 62nd IEEE Conference on Decision and Control (CDC)},
  pages={2995--3002},
  year={2023},
  organization={IEEE}
}

@article{scherer2021convex,
  title={Convex synthesis of accelerated gradient algorithms},
  author={Scherer, Carsten and Ebenbauer, Christian},
  journal={SIAM Journal on Control and Optimization},
  volume={59},
  number={6},
  pages={4615--4645},
  year={2021},
  publisher={SIAM}
}

@article{lessard2016analysis,
  title={Analysis and design of optimization algorithms via integral quadratic constraints},
  author={Lessard, Laurent and Recht, Benjamin and Packard, Andrew},
  journal={SIAM Journal on Optimization},
  volume={26},
  number={1},
  pages={57--95},
  year={2016},
  publisher={SIAM}
}

@article{lessard2022analysis,
  title={The analysis of optimization algorithms: A dissipativity approach},
  author={Lessard, Laurent},
  journal={IEEE Control Systems Magazine},
  volume={42},
  number={3},
  pages={58--72},
  year={2022},
  publisher={IEEE}
}

@book{boyd2004convex,
  title={Convex optimization},
  author={Boyd, Stephen and Vandenberghe, Lieven},
  year={2004},
  publisher={Cambridge University Press}
}

@book{ruszczynski2011nonlinear,
  title={Nonlinear Optimization},
  author={Ruszczynski, Andrzej},
  year={2011},
  publisher={Princeton University Press}
}

@article{choi2024non,
  title={Non-ergodic linear convergence property of the delayed gradient descent under the strongly convexity and the Polyak--{\L}ojasiewicz condition},
  author={Choi, Hyung Jun and Choi, Woocheol and Seok, Jinmyoung},
  journal={Analysis and Applications},
  volume={22},
  number={06},
  pages={1023--1051},
  year={2024},
  publisher={World Scientific}
}

@article{van2017fastest,
  title={The fastest known globally convergent first-order method for minimizing strongly convex functions},
  author={Van Scoy, Bryan and Freeman, Randy A and Lynch, Kevin M},
  journal={IEEE Control Systems Letters},
  volume={2},
  number={1},
  pages={49--54},
  year={2017},
  publisher={IEEE}
}

@article{boczar2017exponential,
  title={Exponential stability analysis via integral quadratic constraints},
  author={Boczar, Ross and Lessard, Laurent and Packard, Andrew and Recht, Benjamin},
  journal={arXiv preprint arXiv:1706.01337},
  year={2017}
}

@inproceedings{li2024generalization,
  title={On the Generalization of the Multivariable {Popov} Criterion for Slope-Restricted Nonlinearities},
  author={Li, Mengmou and Hatanaka, Takeshi and Nagahara, Masaaki},
  booktitle={2024 63rd IEEE Conference on Decision and Control (CDC)},
year = {2024},
note = {to appear.}
}

@book{bauschke2017convex,
  title={Convex Analysis and Monotone Operator Theory in Hilbert Spaces},
  author={Bauschke, Heinz H and Combettes, Patrick L},
  year={2017},
  publisher={Springer}
}

@article{carrasco2013equivalence,
  title={Equivalence between classes of multipliers for slope-restricted nonlinearities},
  author={Carrasco, Joaquin and Heath, William P and Lanzon, Alexander},
  journal={Automatica},
  volume={49},
  number={6},
  pages={1732--1740},
  year={2013},
  publisher={Elsevier}
}

@article{gyotoku2024dual,
  title={On Dual of {LMIs} for Absolute Stability Analysis of Nonlinear Feedback Systems with Static {O'Shea--Zames--Falb} Multipliers},
  author={Gyotoku, Hibiki and Yuno, Tsuyoshi and Ebihara, Yoshio and Magron, Victor and Peaucelle, Dimitri and Tarbouriech, Sophie},
  journal={arXiv preprint arXiv:2411.14339},
  year={2024}
}

@article{su2023necessity,
  title={On the necessity and sufficiency of discrete-time {O’Shea--Zames--Falb} multipliers},
  author={Su, Lanlan and Seiler, Peter and Carrasco, Joaquin and Khong, Sei Zhen},
  journal={Automatica},
  volume={150},
  pages={110872},
  year={2023},
  publisher={Elsevier}
}

@article{seiler2014stability,
  title={Stability analysis with dissipation inequalities and integral quadratic constraints},
  author={Seiler, Peter},
  journal={IEEE Transactions on Automatic Control},
  volume={60},
  number={6},
  pages={1704--1709},
  year={2014},
  publisher={IEEE}
}

@article{carrasco2016zames, 
title={{Zames--Falb} multipliers for absolute stability: From {O'Shea's} contribution to convex searches}, 
author={Carrasco, Joaquin and Turner, Matthew C and Heath, William P}, 
journal={European Journal of Control}, 
volume={28}, 
pages={1--19}, 
year={2016}, 
publisher={Elsevier}
}

@inproceedings{taylor2017performance,
  title={Performance estimation toolbox ({PESTO}): Automated worst-case analysis of first-order optimization methods},
  author={Taylor, Adrien B and Hendrickx, Julien M and Glineur, Fran{\c{c}}ois},
  booktitle={2017 IEEE 56th Annual Conference on Decision and Control (CDC)},
  pages={1278--1283},
  year={2017},
  organization={IEEE}
}

@article{taylor2017exact,
  title={Exact worst-case performance of first-order methods for composite convex optimization},
  author={Taylor, Adrien B and Hendrickx, Julien M and Glineur, Fran{\c{c}}ois},
  journal={SIAM Journal on Optimization},
  volume={27},
  number={3},
  pages={1283--1313},
  year={2017},
  publisher={SIAM}
}

@article{taylor2017smooth,
  title={Smooth strongly convex interpolation and exact worst-case performance of first-order methods},
  author={Taylor, Adrien B and Hendrickx, Julien M and Glineur, Fran{\c{c}}ois},
  journal={Mathematical Programming},
  volume={161},
  pages={307--345},
  year={2017},
  publisher={Springer}
}

@article{li2025first,
  title={First-Order Projected Algorithms With the Same Linear Convergence Rate Bounds as Their Unconstrained Counterparts},
  author={Li, Mengmou and Lestas, Ioannis and Nagahara, Masaaki},
  journal={arXiv preprint arXiv:2503.13965},
  year={2025}
}

@article{scherer2025tutorial,
  title={A Tutorial on Convex Design of Optimization Algorithms by Integral Quadratic Constraints},
  author={Scherer, Carsten W and Ebenbauer, Christian},
  journal={Annual Review of Control, Robotics, and Autonomous Systems},
  volume={9},
  year={2025},
  publisher={Annual Reviews}
}

@inproceedings{van2022absolute,
  title={Absolute stability via lifting and interpolation},
  author={Van Scoy, Bryan and Lessard, Laurent},
  booktitle={2022 IEEE 61st Conference on Decision and Control (CDC)},
  pages={6217--6223},
  year={2022},
  organization={IEEE}
}

@article{carrasco2019convex,
  title={Convex searches for discrete-time {Zames--Falb} multipliers},
  author={Carrasco, Joaquin and Heath, William P and Zhang, Jingfan and Ahmad, Nur Syazreen and Wang, Shuai},
  journal={IEEE Transactions on Automatic Control},
  volume={65},
  number={11},
  pages={4538--4553},
  year={2019},
  publisher={IEEE}
}

@article{li2023distributed,
  title={Distributed optimal secondary frequency control in power networks with delay independent stability},
  author={Li, Mengmou and Watson, Jeremy and Lestas, Ioannis},
  journal={IEEE Transactions on Automatic Control},
  volume={69},
  number={6},
  pages={3748--3763},
  year={2023},
  publisher={IEEE}
}

@article{kulkarni2002all,
  title={All multipliers for repeated monotone nonlinearities},
  author={Kulkarni, Vishwesh V and Safonov, Michael G},
  journal={IEEE Transactions on Automatic Control},
  volume={47},
  number={7},
  pages={1209--1212},
  year={2002},
  publisher={IEEE}
}

@misc{Scherer2023AlgorithmSynthesis,
  author       = {Scherer, Carsten W.},
  title        = {{CarstenScherer/Algorithm-Synthesis}},
  year         = {2023},
  month        = sep,
  version      = {v2},
  howpublished = {\url{https://zenodo.org/badge/latestdoi/691960972}},
  note         = {Software, released 2023-09-15}
}

@article{molzahn2017survey,
  title={A survey of distributed optimization and control algorithms for electric power systems},
  author={Molzahn, Daniel K and D{\"o}rfler, Florian and Sandberg, Henrik and Low, Steven H and Chakrabarti, Sambuddha and Baldick, Ross and Lavaei, Javad},
  journal={IEEE Transactions on Smart Grid},
  volume={8},
  number={6},
  pages={2941--2962},
  year={2017},
  publisher={IEEE}
}

@article{gower2020variance,
  title={Variance-reduced methods for machine learning},
  author={Gower, Robert M and Schmidt, Mark and Bach, Francis and Richt{\'a}rik, Peter},
  journal={Proceedings of the IEEE},
  volume={108},
  number={11},
  pages={1968--1983},
  year={2020},
  publisher={IEEE}
}

@article{li2020accelerated,
  title={Accelerated first-order optimization algorithms for machine learning},
  author={Li, Huan and Fang, Cong and Lin, Zhouchen},
  journal={Proceedings of the IEEE},
  volume={108},
  number={11},
  pages={2067--2082},
  year={2020},
  publisher={IEEE}
}

@article{zhang2022zames,
  title={{Zames--Falb} multipliers for convergence rate: motivating example and convex searches},
  author={Zhang, Jingfan and Seiler, Peter and Carrasco, Joaquin},
  journal={International Journal of Control},
  volume={95},
  number={3},
  pages={821--829},
  year={2022},
  publisher={Taylor \& Francis}
}

@article{d2021acceleration,
  title={Acceleration methods},
  author={d’Aspremont, Alexandre and Scieur, Damien and Taylor, Adrien},
  journal={Foundations and Trends in Optimization},
  volume={5},
  number={1-2},
  pages={1--245},
  year={2021},
  publisher={Emerald Publishing Limited}
}

@article{miller2026structure,
  title={Structure, Analysis, and Synthesis of First-Order Algorithms},
  author={Miller, Jared and Scherer, Carsten and Jakob, Fabian and Iannelli, Andrea},
  journal={arXiv preprint arXiv:2603.24795},
  year={2026}
}

\end{document}